\documentclass[12pt,a4paper,reqno]{amsart}

\usepackage{exscale}
\usepackage{fancyhdr}
\usepackage[centertags]{amsmath}
\usepackage{amsfonts}
\usepackage{amsmath, amsthm}
\usepackage{ulem}
\usepackage{epsfig}

\renewcommand{\sectionmark}[1]
                    {\markboth{Kapitel \thesection\ #1}{}}
\renewcommand{\sectionmark}[1]
                 {\markright{} }

\newtheorem{thm}{Theorem}

\newtheorem{lemma}[thm]{Lemma}
\newtheorem{corollary}[thm]{Corollary}

\newtheorem{proposition}[thm]{Proposition}
\newtheorem*{claim*} {Claim}
\newtheorem{fremdersatz}{Theorem}

\theoremstyle{definition}
   \newtheorem*{definition*}{Definition}

\numberwithin{equation}{section}

\def\parmod{\parskip=2pt plus1pt minus1pt}

\newenvironment{einr}{\parmod
                      \begin{list}{}
                        {\setlength{\rightmargin}{0cm}
                         \setlength{\leftmargin}{0,75cm}
                         \setlength{\labelwidth}{0cm}
                         \setlength{\parsep}{1pt}
                         \setlength{\itemsep}{1pt}
                         \setlength{\topsep}{1pt}
                         \setlength{\partopsep}{0pt}
                         \setlength{\labelsep}{0cm}
                         \setlength{\listparindent}{0pt}
                         \setlength{\itemindent}{0pt}}
                      \item[] \ignorespaces}
                     {\unskip \end{list}}

\def\nat{{\rm I\! N}}

\def\re{{\mathbb R}}

\def\r{\right}
\def\gl{\left\{}
\def\gr{\right\}}
\def\kl{\left(}
\def\kr{\right)}

\def\In{\subseteq}

\def\mi{\setminus}
\def\abb{\longrightarrow}
\def\sym{{\rm Sym}\,}

\renewcommand{\rho}{\varrho}
\renewcommand{\phi}{\varphi}
\renewcommand{\epsilon}{\varepsilon}

\def\beq{\begin{equation}}
\def\eeq{\end{equation}}
\def\beqar{\begin{eqnarray}}
\def\eeqar{\end{eqnarray}}
\def\beqaro{\begin{eqnarray*}}
\def\eeqaro{\end{eqnarray*}}
\def\bsat{\begin{thm}}
\def\esat{\end{thm}}
\def\bsato{\begin{fremdersatz}}
\def\esato{\end{fremdersatz}}
\def\blem{\begin{lemma}}
\def\elem{\end{lem}}
\def\bkor{\begin{corollary}}
\def\ekor{\end{corollary}}
\def\bprop{\begin{proposition}}
\def\eprop{\end{proposition}}
\def\bdefin{\begin{definition}}
\def\edefin{\end{definition}}
\def\bbew{\begin{proof}}
\def\ebew{\end{proof}}
\def\beinr{\begin{einr}}
\def\eeinr{\end{einr}}

\renewcommand{\rho}{\varrho}
\renewcommand{\phi}{\varphi}

\begin{document}

\thispagestyle{plain}

\fancyhead[CE]{J\"urgen Grahl} \fancyhead[CO]{Conjecture of
Br\"uck}
\fancyhead[LE,RO]{\thepage}
\fancyhead[LO,RE]{}
\fancyfoot[CE,CO]{}

\title[Riemann's theorem about conditionally convergent series]{On Riemann's Theorem About\\ Conditionally Convergent Series}

\author{J\"urgen Grahl}
\address{Department of Mathematics, University of W\"urzburg,
  W\"urzburg, Germany} \email{grahl@mathematik.uni-wuerzburg.de}

\author{Shahar Nevo}
\address{Department of Mathematics, Bar-Ilan University,
Ramat-Gan 52900, Israel} \email{nevosh@macs.biu.ac.il}

\thanks{This research is part of the European Science Foundation Networking
Programme HCAA and was supported by  Israel Science Foundation
Grant 395/07.}

\begin{abstract}

We extend Riemann's rearrangement theorem on conditionally
convergent  series of real numbers to multiple instead of simple
sums.
\end{abstract}



\maketitle
\section{Introduction and statement of results}

By a well-known theorem due to B.~Riemann, each conditionally
convergent series of real numbers can be rearranged in such a way
that the new series converges to some arbitrarily given real value
or to $\infty$ or $-\infty$ (see, for example, \cite[\S{
32}]{heuser}). As to series of vectors in $\re^n$, in 1905
P.~L\'evy \cite{levy} and in 1913 E.~Steinitz \cite{steinitz}
showed the following interesting extension (see also
\cite{rosenthal} for a simplified proof).

\bsato{\bf (L\'evy-Steinitz Theorem)}
The set of all sums of rearrangements of a given series of vectors in
$\re^n$ is either the empty set or a translate of a linear subspace
(i.e., a set of the form $v+M$ where $v$ is a given vector and $M$ is
a linear subspace).
\esato

Here, of course, $M$ is the zero space if and only if the series
is absolutely convergent. For a further generalization of the
L\'evy-Steinitz theorem to spaces of infinite dimension, see
\cite{sofi}.

In this paper, we extend Riemann's result in a different
direction, turning from simple to multiple sums which provides
many more possibilities of rearranging a given sum. First of all,
we have to introduce some notations.

By $\sym(n)$ we denote the symmetric group of the set
$\gl1,\dots,n\gr$, i.e., the group of all permutations of
$\gl1,\dots,n\gr$.

If $(a_m)_m$ is a sequence of elements of a non-empty set $X$,
 $J$ is an infinite subset of $\nat^n$ and if $\tau:J\abb
\nat$ is a bijection and
$$b(j_1,\dots,j_n):=a_{\tau(j_1,\dots,j_n)}\qquad \mbox{ for each }
(j_1,\dots,j_n)\in J,$$
then we say that the mapping
$b:J\abb X, \quad (j_1,\dots,j_n)\mapsto b(j_1,\dots,j_n)$
is a rearrangement of $(a_m)_m$. We write
$$\kl b(j_1,\dots,j_n)\;\Bigm|\; (j_1,\dots,j_n)\in J\kr$$
for such a rearrangement (which is a more convenient notation for
our purposes than the notation $\kl
b_{j_1,\dots,j_n}\kr_{(j_1,\dots,j_n)\in J}$ one would probably
expect). Instead of $\big( b(j_1,\dots,j_n)\;\big|\;
(j_1,\dots,j_n)\in \nat^n\big)$, we also write $\big(
b(j_1,\dots,j_n)\; \big|\; j_1,\dots,j_n\ge 1\big)$ and also use
notations like $\big( b(j_1,\dots,j_n)\; \big|\; j_1\ge
k_1,\dots,j_n\ge k_n\big)$ which should be self-explanatory now.

\bigskip

With these notations, we can state our main result as follows.

\bsat{}\label{mainresult} Let $n\ge2$ be a natural number and let
$\sum_{m=1}^{\infty} a_m$ be a conditionally convergent series of
real numbers $a_m$. For each $\sigma\in \sym(n)$, let $\big(
s_k^{(\sigma)}\big)_{k\ge1}$ be a sequence of real numbers. Then
there exists a rearrangement $\kl b(j_1,\dots,j_n)\;|\;
j_1,\dots,j_n\ge 1\kr$ of $(a_m)_m$ such that for each $\sigma\in
\sym(n)$ and each $k\ge 1$, one has
$$\sum_{j_1=1}^{k}\sum_{j_2=1}^{\infty}\dots\sum_{j_n=1}^{\infty}
b(j_{\sigma(1)},\dots,j_{\sigma(n)})
=s_k^{(\sigma)}.$$
\esat

\bkor\label{Cor} Let $n\ge1$ be a natural number and let
$\sum_{m=1}^{\infty} a_m$ be a conditionally convergent series of
real numbers $a_m$. For each $\sigma\in \sym(n)$, let
$s^{(\sigma)}$ be a real number or $\pm\infty$. Then there exists
a rearrangement $\kl b(j_1,\dots,j_n)\;|\; j_1,\dots,j_n\ge 1\kr$
of $(a_m)_m$ such that for each $\sigma\in \sym(n)$, one has
$$\sum_{j_1=1}^{\infty}\sum_{j_2=1}^{\infty}\dots\sum_{j_n=1}^{\infty}
b(j_{\sigma(1)},\dots,j_{\sigma(n)})=s^{(\sigma)}.$$
\ekor

\bbew{} For $n=1$, this is just Riemann's theorem. For $n\ge2$, it
is an immediate consequence of Theorem \ref{mainresult}. \ebew


By moving to continuous functions on $\re^n$, we can construct an
example of a continuous function in the ``positive part''
$Q:=[0,\infty)^n$ of $\re^n$ whose iterated integrals exist for each
order of integration, but all of them have different values. This is a
kind of ``ultimate'' counterexample to show that the assumptions in
Fubini's theorem are inevitable.

\bkor{}\label{Fubini} Let $n\ge 2$ be a natural number. For each
$\sigma\in \sym(n)$, let $s^{(\sigma)}$ be a real number or
$\pm\infty$. Then there exists a function $f\in C^\infty(Q)$ such
that \beq\label{fubini1} \int_0^\infty\dots \int_0^\infty
f(x_1,\dots,x_n) \;dx_{\sigma(1)}
\;dx_{\sigma(2)}\dots\;dx_{\sigma(n)}=s^{(\sigma)} \quad \mbox{
for each
  } \sigma\in \sym(n).
\eeq
\ekor

\bbew{} Let $\sum_{m=1}^{\infty} a_m$ be some conditionally
convergent series. By Corollary \ref{Cor}, there exists a
rearrangement $\kl b(j_1,\dots,j_n)\;|\; j_1,\dots,j_n\ge 1\kr$ of
$(a_m)_m$ such that for each $\sigma\in \sym(n)$, one has
$$\sum_{k_n=1}^{\infty}\dots\sum_{k_1=1}^{\infty}
b(k_{\sigma^{-1}(1)},\dots,k_{\sigma^{-1}(n)})=s^{(\sigma)}.$$
We set $I=[-0.49\,;\,0.49]^n$ and define the function $\varphi:\re^n\abb\re$ by
$$\varphi(x):=\gl
\begin{array}{rl} A e^{-1/(0.49-||x||)^2} & \mbox{ for }
||x||<0.49 \\ 0 & \mbox{ for } ||x||\ge 0.49,\end{array}\r. $$
where $A>0$ and $||.||$ is the Euclidean norm on $\re^n$. Then
$\varphi\in C^\infty(\re^n)$, and $\varphi$ vanishes outside the
compact set $I$. So $\varphi$ is integrable with respect to the
Lebesgue measure $\lambda$, and by choosing an appropriate $A$ we
can obtain
$$\int_{\re^n}\varphi(x)\;d\lambda(x)=1.$$
In particular, by Fubini's theorem the last integral can be written in
any order of integration, i.e.
$$\int_{-\infty}^\infty\dots \int_{-\infty}^\infty \varphi(x_1,\dots,x_n) \;dx_{\sigma(1)}
\;dx_{\sigma(2)}\dots\;dx_{\sigma(n)}=1 \qquad \mbox{ for each
  } \sigma\in \sym(n).$$
Since $\varphi$ vanishes outside $I$, for any $j_1,\dots, j_n\ge
1$, we also have \beq\label{Fubini2} \int_{-j_n}^\infty\dots
\int_{-j_1}^\infty \varphi(x_1,\dots,x_n) \;dx_{\sigma(1)}
\;dx_{\sigma(2)}\dots\;dx_{\sigma(n)}=1 \qquad \mbox{
  for each   } \sigma\in \sym(n).
\eeq
Now we define $f:Q\abb\re$ by
$$f(x_1,\dots,x_n):=\sum_{j_1=1}^{\infty}\sum_{j_2=1}^{\infty}\dots\sum_{j_n=1}^{\infty}
b(j_1,\dots,j_n)\cdot\varphi(x_1-j_1,\dots,x_n-j_n).$$ For each
$x=(x_1,\dots,x_n)\in Q$, at most one of the terms
$\varphi(x_1-j_1,\dots,x_n-j_n)$ is non-zero, so the multiple sum
in the definition of $f$ reduces to just one term, and we conclude
that $f\in C^\infty(Q)$. Let  $\sigma\in\sym(n)$ be given. Then we
obtain by (\ref{Fubini2})
\begin{align*}
& \int_0^\infty\dots \int_0^\infty f(x_1,\dots,x_n)
\;dx_{\sigma(1)}
\;dx_{\sigma(2)}\dots\;dx_{\sigma(n)}\\
&\qquad=
\sum_{j_{\sigma(n)}=1}^{\infty}\dots\sum_{j_{\sigma(1)}=1}^{\infty}
b(j_1,\dots,j_n)
\int_0^\infty\dots \int_0^\infty \varphi(x_1-j_1,\dots,x_n-j_n)\;dx_{\sigma(1)}\dots\;dx_{\sigma(n)}\\
&\qquad=\sum_{k_n=1}^{\infty}\dots\sum_{k_1=1}^{\infty}b(k_{\sigma^{-1}(1)},\dots,k_{\sigma^{-1}(n)})\cdot 1 \\
&\qquad=s^{(\sigma)}, \end{align*} hence (\ref{fubini1}). \ebew

Observe that in the case $n=2$, by Corollary \ref{Fubini} we get
the existence of a function $f\in C^\infty([0,\infty)^2)$ such
that
$$\int_{0}^{\infty}\int_{0}^{\infty} f(x,y)\;dx\,dy =+\infty
\qquad\mbox{ and } \qquad \int_{0}^{\infty}\int_{0}^{\infty}
f(x,y)\;dy\,dx =-\infty.$$ For the functions $f$ from Corollary
\ref{Fubini}, in general, the improper integral $\int_Q
f(x_1,\dots,x_n) d(x_1,\dots,x_n)$ (in the sense of Riemann) does
not exist in the extended sense\footnote{We say that the improper
  integral $\int_Q f(x_1,\dots,x_n) \;d(x_1,\dots,x_n)$ exists in the
  extended sense if for arbitrary exhaustions $(K_m)_m$ of $Q$ with
  compact sets $K_m$, the limits $\lim_{m\to\infty} \int_{K_m}
  f(x_1,\dots,x_n) \;d(x_1,\dots,x_n)$ exist and are equal.}. A
necessary condition for the existence of this integral is that
$s^{(\sigma)}=s^{(\tau)}$ for every $\sigma,\tau\in \sym(n)$. However,
it can be shown that this condition is not sufficient for the
convergence of the improper integral.

It is obvious that, by modifying the definition of $f$ (such that
its ``peaks'' are at the points
$\kl\frac{1}{2^{j_1}},\dots\frac{1}{2^{j_n}}\kr$ rather than at
the points $(j_1,\dots,j_n)$), one can replace $Q$ by $(0,1]^n$ in
Corollary~\ref{Fubini}, i.e., we can find a function $f\in
C^\infty((0,1]^n)$ whose iterated integrals exist for every order
of integration, but each time give different values. Of course,
this is not possible for the compact cube $[0,1]^n$, since
continuous functions on compact sets are Lebesgue-integrable, so
by Fubini's Theorem their integrals are independent of the order
of integration.

\section{Proofs}

It is well known that a {\it convergent} series
$\sum_{m=1}^{\infty} a_m$ of real numbers is conditionally
convergent if and only if \beq\label{CondConvergable} \sum_{a_m>0}
a_m=\infty \qquad\mbox{ and } \qquad \sum_{a_m<0} a_m=-\infty.
\eeq This property is a bit more general than the property of
conditional convergence: It may also hold for series which are not
convergent themselves. It turns out that this is the property we
actually deal with in the proof of our main result. This gives
rise to the following definition.

\begin{definition*} We say that a series $\sum_{m=1}^{\infty} a_m$ of real
numbers is {\bf
  conditionally convergable} if $\lim_{m\to\infty} a_m=0$ and if
(\ref{CondConvergable}) holds. \end{definition*}

As the proof of Riemann's theorem shows, a series is conditionally
convergable if and only if it has some rearrangement which is
conditionally convergent.

The main advantage of this newly introduced notion is the following:
Conditional convergability is invariant under rearrangements while
conditional convergence is not.

\begin{lemma} \label{split} Let $\sum_{m=1}^{\infty} a_m$ be a
conditionally convergable series of real numbers $a_m$. Then there
is a disjoint partition $\nat=\bigcup_{t=1}^{\infty} I_t$ of $\nat$
into infinite subsets $I_t$ such that for each $t\in\nat$ the series
$\sum_{m\in I_t} a_m$ is conditionally convergable\footnote{In
  notations like $\sum_{j\in I_t} a_j$, the order of   summation is of
  course understood to be in the natural order of increasing indices
  $j$. On the other hand, since conditional convergability is
  invariant under rearrangements, we do not have to specify the order
  of summation at all, at least not for the purpose of Lemma \ref{split}.}.
\end{lemma}

\bbew{}
I. Let $(\beta_m)_m$ be a sequence of {\it non-negative} numbers such that
$$\sum_{m=1}^{\infty} \beta_m=\infty.$$
Then it is evident that one can decompose $\nat$ into two infinite
disjoint subsets $I_1,I^{(2)}$ such that $1\in I_1$ and
$$\sum_{m\in I_1} \beta_m=\infty \qquad \mbox{ and } \qquad
\sum_{m\in I^{(2)}} \beta_m=\infty.$$
Let us assume that we have already found subsets $I_1,\dots,I_t,
I^{(t+1)} \In\nat$ such that $\nat=I_1\cup\dots\cup I_t\cup I^{(t+1)}$
is a disjoint union,
$$\sum_{m\in I_s} \beta_m=\infty \quad (s=1,\dots,t)\qquad \mbox{ and } \qquad
\sum_{m\in I^{(t+1)}} \beta_m=\infty$$
and such that $\min(\nat\mi (I_1\cup\dots\cup I_{s-1}))\in I_s$ for
$s=1,\dots,t$.
Then we can find a disjoint decomposition $I^{(t+1)}=I_{t+1}\cup I^{(t+2)}$ such
that
$$\sum_{m\in I_{t+1}} \beta_m=\infty \qquad \mbox{ and } \qquad
\sum_{m\in I^{(t+2)}} \beta_m=\infty$$
and such that $\min(\nat\mi (I_1\cup\dots\cup I_t))\in I_{t+1}$.

In this way, inductively we construct subsets $I_t\In\nat$ such
that $\sum_{m\in I_t} \beta_m=\infty$ for all $t$. It is evident
that $\bigcup_{t=1}^{\infty} I_t=\nat$ and that this union is
disjoint. (Observe that it is crucial to put the smallest element
from $\nat\mi (I_1\cup\dots\cup I_{t-1})$ into $I_t$ in each step,
in order to guarantee that each natural number appears in some
$I_t$, i.e., that it is not forgotten ``forever''.)

II. Let $\sum_{m=1}^{\infty} a_m$ be a conditionally convergable
series of real numbers and let
$$P:=\gl m\in\nat\;|\; a_m\ge 0\gr, \qquad
N:=\gl m\in\nat\;|\; a_m< 0\gr.$$
Then we have
$$\sum_{m\in P} a_m=+\infty, \qquad
\sum_{m\in N} a_m=-\infty.$$
By I. there exist disjoint decompositions $P= \bigcup_{t=1}^{\infty}
P_t$ and $N= \bigcup_{t=1}^{\infty}N_t$ of $P$ and $N$ into infinite
subsets $P_t, N_t$ such that
$$\sum_{m\in P_t} a_m=\infty \qquad \mbox{ and } \qquad
\sum_{m\in N_t} a_m=-\infty$$
for all $t$. If we set
$$I_t:=P_t\cup N_t,$$ then for every $t$ the series $\sum_{m\in I_t}
a_m$ is conditionally convergable, and
$\nat=\bigcup_{t=1}^{\infty} I_t$ is a disjoint decomposition.
This proves the assertion. \ebew

Since the proof of the general case of Theorem \ref{mainresult} is
quite abstract, we   start with a discussion of the case $n=2$ to
give the reader an idea of what is really going on.

\begin{proof}[Proof of the Case $n=2$ of Theorem \ref{mainresult}.] Here,
$\sym(2)$ consists of two elements $\sigma=(1 \quad
2)=id_{\gl1,2\gr}$ and $\tau=(2\quad 1)$.

According to Lemma \ref{split}, there exists a disjoint partition
$\nat=\bigcup_{t=1}^{\infty} I_t$ of $\nat$ into infinite subsets
$I_t$ such that for each $t\in\nat$ the series $\sum_{m\in I_t}
a_m$ is conditionally convergable. By Riemann's theorem, we can
find a rearrangement $(b(1,k)\;|\; k\in\nat)$ of $(a_m)_{m\in
I_1}$ such that
$$\sum_{k=1}^{\infty} b(1,k)=s_1^{(\sigma)}.$$
In the same way, we can find a rearrangement $(b(j,1)\;|\; j\ge2)$ of $(a_m)_{m\in I_2}$ such that
$$\sum_{j=2}^{\infty} b(j,1)=s_1^{(\tau)}-b(1,1).$$
Next, we choose a rearrangement $(b(2,k)\;|\; k\ge2)$ of $(a_m)_{m\in I_3}$ such that
$$\sum_{k=2}^{\infty} b(2,k)=s_2^{(\sigma)}-s_1^{(\sigma)}-b(2,1)$$
and a rearrangement $(b(j,2)\;|\; j\ge3)$ of $(a_m)_{m\in I_4}$ such that
$$\sum_{j=3}^{\infty} b(j,2)=s_2^{(\tau)}-s_1^{(\tau)}-b(1,2)-b(2,2),$$
and so on. Proceeding in this way, for each $j\ge 2$ we find a
rearrangement $(b(j,k)\;|\; k\ge j)$ of $(a_m)_{m\in I_{2j-1}}$ such that
$$\sum_{k=j}^{\infty} b(j,k)=s_j^{(\sigma)}-s_{j-1}^{(\sigma)}-\sum_{k=1}^{j-1} b(j,k),$$
and for each $k\ge 2$ we find a rearrangement $(b(j,k)\;|\; j\ge k+1)$ of $(a_m)_{m\in I_{2k}}$ such that
$$\sum_{j=k+1}^{\infty} b(j,k)=s_k^{(\tau)}-s_{k-1}^{(\tau)}-\sum_{j=1}^{k} b(j,k).$$
In this way, $b(j,k)$ is uniquely defined for all $j,k\in\nat$,
$(b(j,k)\;|\; j,k\in\nat)$ is a rearrangement of $(a_m)_m$, and the
$b(j,k)$ satisfy the equations
$$\sum_{j=1}^{N}\sum_{k=1}^{\infty} b(j,k)
=s_1^{(\sigma)}+\sum_{j=2}^{N} \kl
s_j^{(\sigma)}-s_{j-1}^{(\sigma)}\kr =s_N^{(\sigma)},$$
$$\sum_{k=1}^{N}\sum_{j=1}^{\infty} b(j,k)
=s_1^{(\tau)}+\sum_{k=2}^{N} \kl s_k^{(\tau)}-s_{k-1}^{(\tau)}\kr =s_N^{(\tau)}$$
for all $N\in\nat$, as asserted.
\end{proof}

Now we turn to the general case.

\begin{proof}[Proof of Theorem \ref{mainresult}] We prove the theorem by
induction. It suffices to show that for each $n\ge2$, the validity
of Corollary \ref{Cor} for $n-1$ implies the validity of the
theorem for $n$. (Here it is important to note that the corollary
also holds for $n=1$ in view of Riemann's theorem.)

So let some $n\ge 2$ be given and assume that Corollary \ref{Cor} is
valid for $n-1$ instead of $n$. Let $(a_m)_m$ be a sequence of real
numbers such that $\sum_{m=1}^{\infty} a_m$ is conditionally convergent.

According to Lemma \ref{split}, there exists a disjoint partition
$\nat=\bigcup_{t=1}^{\infty} I_t$ of $\nat$ into infinite subsets
$I_t$ such that for each $t\in\nat$ the series $\sum_{m\in I_t}
a_m$ is conditionally convergable.

\medskip

For an integer $d\ge 0$, we consider the following assumption.

\medskip

\noindent{\bf Assumption $A_d$.} The quantities $b(j_1,\dots,j_n)$
are already defined for all $j_1,\dots,j_n\in\nat$ with $\gl
j_1,\dots,j_n\gr\cap\gl 1,\dots,d\gr\ne\emptyset$ such that
$$\kl b(j_1,\dots,j_n)\;|\; \gl j_1,\dots,j_n\gr\cap\gl
1,\dots,d\gr\ne\emptyset\kr$$ is a rearrangement of $\kl a_m\;|\;
m\in\bigcup_{t=1}^{nd} I_t\kr$ and such that for all
$k\in\gl1,\dots,d\gr$, all $\nu\in\gl1,\dots,n\gr$ and all
$\sigma\in\sym(n)$ with $\sigma(\nu)=1$ one has
\beq\label{AssumptionA}
\sum_{j_2=1}^{\infty}\dots\sum_{j_n=1}^{\infty}
b(j_{\sigma(1)},\dots,j_{\sigma(\nu-1)},k,j_{\sigma(\nu+1)},\dots,j_{\sigma(n)})
=s_k^{(\sigma)}- s_{k-1}^{(\sigma)}; \eeq here, $s_0^{(\sigma)}=0$
for all $\sigma\in \sym(n)$.

Here, for $\nu=1$, the quantity
$b(j_{\sigma(1)},\dots,j_{\sigma(\nu-1)},k,j_{\sigma(\nu+1)},\dots,j_{\sigma(n)})$
is of course understood to be just
$b(k,j_{\sigma(2)},\dots,j_{\sigma(n)})$. A similar comment
applies to several other notations in the sequel.

We note that this is trivially satisfied for $d=0$ since in this case
the assumption is empty.

Now let some integer $d\ge 0$ be given and assume that $A_d$ is
satisfied. We want to show that also $A_{d+1}$ is satisfied.  This
is done by induction once again: For given
$\mu\in\gl1,\dots,n+1\gr$, we consider the following assumption.

\medskip

\noindent{\bf Assumption $B_{d,\mu}$.} The quantities
$b(j_1,\dots,j_n)$ are already defined for all
$j_1,\dots,j_n\in\nat$ with $d+1\in\gl j_1,\dots,j_{\mu-1}\gr$
such that
$$\kl b(j_1,\dots,j_n)\;|\; j_1,\dots,j_n\ge d+1,\; d+1\in\gl j_1,\dots,j_{\mu-1}\gr\kr$$
is a rearrangement of $\kl a_m\;|\;
m\in\bigcup_{t=nd+1}^{nd+\mu-1} I_t\kr$ and such that for all
$\nu\in\gl1,\dots,\mu-1\gr$ and all $\sigma\in\sym(n)$ with
$\sigma(\nu)=1$, one has \beq\label{AssumptionB}
\sum_{j_2=1}^{\infty}\dots\sum_{j_n=1}^{\infty}
b(j_{\sigma(1)},\dots,j_{\sigma(\nu-1)},d+1,j_{\sigma(\nu+1)},\dots,j_{\sigma(n)})
=s_{d+1}^{(\sigma)}-s_d^{(\sigma)}. \eeq Again we note that for
$\mu=1$ the assumption $B_{d,\mu}$ is empty, hence trivially true.

So we let some $\mu\in\gl1,\dots,n\gr$ be given and assume that
$B_{d,\mu}$ holds. For $\sigma\in\sym(n),$ we set
$$\delta(\sigma,\nu):=\gl\begin{array}{ll}
d+2 & \mbox{ if }\nu\in\gl \sigma(1),\dots,\sigma(\mu-1)\gr,\\
d+1 & \mbox{ if } \nu\in\gl
\sigma(\mu+1),\dots,\sigma(n)\gr.\end{array}\r.$$ It is not needed
to define $\delta(\sigma,\sigma(\mu))$  as we will see in the
sequel.

\begin{claim*} For all $l=2,\dots,n$ and all $\sigma\in\sym(n)$ with
$\sigma(\mu)=1$, the series \beq\label{Konvergenz}
\sum_{j_2=1}^{\infty}\dots\sum_{j_{l-1}=1}^{\infty}\sum_{j_l=1}^{\delta(\sigma,l)-1}
\sum_{j_{l+1}=\delta(\sigma,l+1)}^\infty\dots\sum_{j_n=\delta(\sigma,n)}^\infty
b(j_{\sigma(1)},\dots,j_{\sigma(\mu-1)},d+1,j_{\sigma(\mu+1)},\dots,j_{\sigma(n)})
\eeq is (well-defined and) convergent.
\end{claim*}

\begin{proof} Let some $l\in\gl2,\dots,n\gr$ and some
$\sigma\in\sym(n)$ with $\sigma(\mu)=1$ be given. In view of $l\ne
1=\sigma(\mu)$ we have to consider only the following two cases.

\noindent{\bf Case 1:} $l\in\gl \sigma(1),\dots,\sigma(\mu-1)\gr.$

Then $\delta(\sigma,l)-1=d+1$ and there is some
$\lambda\in\gl1,\dots,\mu-1\gr$ such that $l=\sigma(\lambda)$. Now we define a
permutation $\tau\in\sym(n)$ as follows:
\beq\label{Permut}
\tau(i):=\sigma(i) \quad\mbox{ for } i\ne \lambda,\mu, \qquad
\tau(\lambda):=\sigma(\mu)=1, \qquad
\tau(\mu):=\sigma(\lambda).
\eeq
The series (\ref{Konvergenz}) is the sum of the
$\delta(\sigma,l)-1=d+1$ series
$$\sum_{j_2=1}^{\infty}\dots\sum_{j_{l-1}=1}^{\infty}
\sum_{j_{l+1}=\delta(\sigma,l+1)}^\infty\dots\sum_{j_n=\delta(\sigma,n)}^\infty
b(j_{\tau(1)},\dots,j_{\tau(\lambda-1)},j_l,j_{\tau(\lambda+1)},\dots,
j_{\tau(\mu-1)},d+1,j_{\tau(\mu+1)},\dots,j_{\tau(n)})$$ where
$j_l=1,\dots,d+1$. This series is convergent by assumption
$B_{j_l-1,\lambda+1}$ (see (\ref{AssumptionB})).
This shows the convergence of the series in (\ref{Konvergenz}).

\noindent{\bf Case 2:} $l\in\gl \sigma(\mu+1),\dots,\sigma(n)\gr.$

Then $\delta(\sigma,l)-1=d$ and there is some
$\lambda\in\gl\mu+1,\dots,n\gr$ such that $l=\sigma(\lambda)$. Now we
define $\tau$ as in (\ref{Permut}). The series (\ref{Konvergenz}) is
the sum of the $\delta(\sigma,l)-1=d$ series
$$\sum_{j_2=1}^{\infty}\dots\sum_{j_{l-1}=1}^{\infty}
\sum_{j_{l+1}=\delta(\sigma,l+1)}^\infty\dots\sum_{j_n=\delta(\sigma,n)}^\infty
b(j_{\tau(1)},\dots,j_{\tau(\mu-1)},d+1,j_{\tau(\mu+1)},\dots,
j_{\tau(\lambda-1)},j_l,j_{\tau(\lambda+1)},\dots,j_{\tau(n)})$$
where $j_l=1,\dots,d$. This latter series is convergent by
assumption $A_{j_l}$ (see (\ref{AssumptionA})). So the series in
(\ref{Konvergenz}) is convergent as  well. This proves our claim.
\end{proof}

According to Corollary \ref{Cor}, one can choose
$$\kl b(j_1,\dots,j_n)\;|\; j_1,\dots,j_{\mu-1}\ge
d+2,j_\mu=d+1,j_{\mu+1},\dots,j_n\ge d+1\kr$$ as a rearrangement
of $I_{nd+\mu}$ such that for all $\sigma\in\sym(n)$ with
$\sigma(\mu)=1,$ one has \begin{align*} &
\sum_{j_2=\delta(\sigma,2)}^{\infty}\dots\sum_{j_n=\delta(\sigma,n)}^{\infty}
b(j_{\sigma(1)},\dots,j_{\sigma(\mu-1)},d+1,j_{\sigma(\mu+1)},\dots,j_{\sigma(n)})\\
\\
& = s_{d+1}^{(\sigma)}-s_d^{(\sigma)}\\
&\quad-\sum_{l=2}^{n}
\sum_{j_2=1}^{\infty}\dots\sum_{j_{l-1}=1}^{\infty}\sum_{j_l=1}^{\delta(\sigma,l)-1}
\sum_{j_{l+1}=\delta(\sigma,l+1)}^\infty\dots\sum_{j_n=\delta(\sigma,n)}^\infty
b(j_{\sigma(1)},\dots,j_{\sigma(\mu-1)},d+1,j_{\sigma(\mu+1)},\dots,j_{\sigma(n)}).
\end{align*} Here we have used the claim above (see (\ref{Konvergenz})) and
the fact that we can identify the subset $\gl
\sigma\in\sym(n)\;|\,\sigma(\mu)=1\gr$ with $\sym(n-1)$.

Then one can see that
$$\sum_{j_2=1}^{\infty}\dots\sum_{j_n=1}^{\infty}
b(j_{\sigma(1)},\dots,j_{\sigma(\mu-1)},d+1,j_{\sigma(\mu+1)},\dots,j_{\sigma(n)})
=s_{d+1}^{(\sigma)}-s_d^{(\sigma)}$$
for all $\sigma\in\sym(n)$ with $\sigma(\mu)=1$.

In this way, we have defined $b(j_1,\dots,j_n)$ for all
$j_1,\dots,j_n\in\nat$ with $d+1\in\gl j_1,\dots,j_\mu\gr$ such
that
$$\kl b(j_1,\dots,j_n)\;|\; j_1,\dots,j_n\ge d+1,\; d+1\in\gl j_1,\dots,j_\mu\gr\kr$$
is a rearrangement of $\kl a_m\;|\; m\in\bigcup_{t=nd+1}^{nd+\mu}
I_t\kr$ and such that for all $\nu\in\gl1,\dots,\mu\gr$ and all
$\sigma\in\sym(n)$ with $\sigma(\nu)=1,$ one has
$$\sum_{j_2=1}^{\infty}\dots\sum_{j_n=1}^{\infty}
b(j_{\sigma(1)},\dots,j_{\sigma(\nu-1)},d+1,j_{\sigma(\nu+1)},\dots,j_{\sigma(n)})
=s_{d+1}^{(\sigma)}-s_d^{(\sigma)}.$$
Hence $B_{d,\mu+1}$ holds.

By induction we deduce that $B_{d,n+1}$ holds. But this (together
with assumption $A_d$) just means that $A_{d+1}$ holds. So by
induction, we obtain the validity of $A_d$ for all $d\ge 0$. This
proves our theorem. \ebew


\bibliographystyle{amsplain}

\end{document}